\documentclass[paper=a4,english,fontsize=11pt,parskip=half,abstract=true]{scrartcl}
\usepackage{babel}
\usepackage[utf8]{inputenc}
\usepackage[T1]{fontenc}
\usepackage[left=20mm,right=20mm,top=30mm,bottom=30mm]{geometry}
\usepackage{amsmath}
\usepackage{amsthm}
\usepackage{amssymb}
\usepackage{enumerate} 
\usepackage{thmtools}
\usepackage{mathtools}
\mathtoolsset{centercolon} 
\usepackage[bookmarks=true,
            pdftitle={On a fixed point formula of Navarro--Rizo},
            pdfauthor={Benjamin Sambale},
            pdfkeywords={},
            pdfstartview={FitH}]{hyperref}

\newtheorem{Thm}{Theorem} 

\newtheorem{Prop}[Thm]{Proposition}

\numberwithin{equation}{section}

\allowdisplaybreaks[1]

\renewcommand{\phi}{\varphi}
\newcommand{\C}{\mathrm{C}}
\newcommand{\N}{\mathrm{N}}

\newcommand{\pcore}{\mathrm{O}}

\newcommand{\ZZ}{\mathbb{Z}}
\newcommand{\CC}{\mathbb{C}}

\newcommand{\FF}{\mathbb{F}}

\newcommand{\GL}{\operatorname{GL}}

\newcommand{\PGL}{\operatorname{PGL}}
\newcommand{\PSL}{\operatorname{PSL}}

\newcommand{\Irr}{\operatorname{Irr}}

\newcommand{\Syl}{\operatorname{Syl}}
\newcommand{\Hall}{\operatorname{Hall}}

\title{On a fixed point formula of Navarro--Rizo}
\author{Benjamin Sambale\footnote{Institut für Algebra, Zahlentheorie und Diskrete Mathematik, Leibniz Universität Hannover, Welfengarten 1, 30167 Hannover, Germany,
\href{mailto:sambale@math.uni-hannover.de}{sambale@math.uni-hannover.de}}}
\date{\today}

\begin{document}
\frenchspacing
\maketitle
\begin{abstract}\noindent
Let $G$ be a $\pi$-separable group with a Hall $\pi$-subgroup $H$ or order $n$. For $x\in H$ let $\lambda(x)$ be the number of Hall $\pi$-subgroups of $G$ containing $x$. We show that $\prod_{d\mid n}\prod_{x\in H}\lambda(x^{d})^{\frac{n}{d}\mu(d)}=1$, where  $\mu$ is the Möbius function.
This generalizes fixed point formulas for coprime actions by Brauer, Wielandt and Navarro--Rizo. We further investigate an additive version of this formula.
\end{abstract}

\section{Introduction}

Navarro and Rizo~\cite{NR} proved the following fixed point equation related to formulas of Brauer and Wielandt.

\begin{Thm}[Navarro--Rizo]\label{NR}
Let $P$ be a finite $p$-group acting on a $p'$-group $N$. Then
\begin{equation}\label{eq:NR}
|\C_N(P)|=\Bigl(\prod_{x\in P}\frac{|\C_N(x)|}{|\C_N(x^p)|^{1/p}}\Bigr)^{\frac{p}{|P|(p-1)}}.
\end{equation}
\end{Thm}

For a finite group $G$ and $x\in P\in\Syl_p(G)$ let $\lambda_G(x)$ be the number of Sylow $p$-subgroups of $G$ containing $x$.
In the situation of \autoref{NR}, the Sylow $p$-subgroups of $G:=N\rtimes P$ have the form $nPn^{-1}$ for $n\in N$. Note that $\N_G(P)=\C_N(P)P$. 
Suppose that $x\in P\cap nPn^{-1}$. Then there exists $y\in P$ such that $x=nyn^{-1}$. Since $[n,y]=xy^{-1}\in P\cap N=1$, it follows that $x=y$ and $n\in\C_N(x)$. Hence, $\lambda_G(x)=|\C_N(x):\C_N(P)|$ for all $x\in P$. Now \eqref{eq:NR} turns into the more elegant formula:
\begin{equation}\label{eq:NR2}
\prod_{x\in P}\lambda_G(x^p)=\prod_{x\in P}\lambda_G(x)^p.
\end{equation}

We show that this holds more generally for all $p$-solvable groups. In fact, our main theorem applies to $\pi$-separable groups, where $\pi$ is any set of primes. Recall that a $\pi$-separable group $G$ has a unique conjugacy class of Hall $\pi$-subgroups. For a $\pi$-element $x\in G$ let $\lambda_G(x)$ be the number of Hall $\pi$-subgroups of $G$ containing $x$. 

\begin{Thm}\label{piformula}
Let $G$ be a $\pi$-separable group with a Hall $\pi$-subgroup $H$ of order $n$. Then
\begin{equation}\label{eq:hall}
\prod_{d\,\mid\, n}\Bigl(\prod_{x\in H}\lambda_G\bigl(x^{d}\bigr)^{\frac{n}{d}}\Bigr)^{\mu(d)}=1,
\end{equation}
where $\mu$ is the Möbius function.
\end{Thm}

For $\pi=\{p\}$, \eqref{eq:hall} becomes \eqref{eq:NR2}. We will show in \autoref{cyclic} that \eqref{eq:hall} holds for arbitrary groups whenever they have a cyclic Hall $\pi$-subgroup (by a result of Wielandt, the Hall $\pi$-subgroups are conjugate in this situation too, see \cite[Satz~III.5.8]{Huppert}). 
In general, the left hand side of \eqref{eq:NR2} can be larger or smaller than the right hand side (consider $G=A_5$ and $G=\GL(3,2)$ for $p=2$). We did not find a non-solvable group fulfilling \eqref{eq:NR2} for $p=2$. 

In the last section we obtain the following additive version of \eqref{eq:hall}.

\begin{Thm}\label{additive}
Let $G$ be a finite group with a Hall subgroup $H$ of order $n$. Then
\[\frac{1}{n^2}\sum_{d\,\mid\, n}\mu(d)\sum_{h\in H}\lambda_G(h^d)^{\frac{n}{d}}\]
is a non-negative integer, which is zero if and only if $1\ne H\unlhd G$.
\end{Thm}

\section{The proof of \autoref{piformula}}

In the first step we reduce \autoref{piformula} to $\pi$-nilpotent groups.
Let $\alpha(G)$ be the left hand side of \eqref{eq:hall}. Since
\[\alpha(G)=\Bigl(\prod_{d\,\mid\, n}\prod_{x\in H}\lambda_G(x^d)^{\mu(d)/d}\Bigr)^n,\]
we may replace $n$ by the product of the prime divisors of $|H|$. Suppose that $p\in\pi$ does not divide $n$. Then $\langle x^d\rangle=\langle x^{dp}\rangle$ for all $x\in H$ and $d\mid n$. Since $\lambda_G(x^d)$ only depends on $\langle x^d\rangle$, it follows that
\[\prod_{d\,\mid\, np}\prod_{x\in H}\lambda_G(x^d)^{\frac{np}{d}\mu(d)}=\prod_{d\,\mid\, n}\prod_{x\in H}\lambda_G(x^d)^{\frac{np}{d}\mu(d)}\lambda_G(x^d)^{-\frac{n}{d}\mu(d)}=\alpha(G)^{p-1}.\]
Hence, we can assume that $n=\prod_{p\in\pi}p$. 

Suppose that $N:=\pcore_{\pi}(G)\ne 1$. Since $N$ lies in every Hall $\pi$-subgroup of $G$, we have $\lambda_G(xy)=\lambda_G(x)$ for all $x\in H$ and $y\in N$. Moreover, $\lambda_G(x)=\lambda_{G/N}(xN)$. 
By induction on $|G|$, we obtain
\[\alpha(G)=\prod_{x\in H}\prod_{d\,\mid\, n}\lambda_G(x^d)^{\frac{n}{d}\mu(d)}=\prod_{xN\in H/N}\Bigl(\prod_{d\,\mid\, n}\lambda_{G/N}(x^dN)^{\frac{n}{d}\mu(d)}\Bigr)^{|N|}=\alpha(G/N)^{|N|}=1.\]
Thus, we may assume that $\pcore_{\pi}(G)=1$. Then $N:=\pcore_{\pi'}(G)\ne 1$. By the argument from the introduction, we have $\lambda_{HN}(x)=|\C_N(x):\C_N(H)|$ for $x\in H$.
If $x$ lies in another Hall $\pi$-subgroup $K$, then 
\[\lambda_{HN}(x)=|\C_N(x):\C_N(H)|=|\C_N(x):\C_N(K)|=\lambda_{KN}(x),\] 
because $H$ and $K$ are conjugate in $G$. It follows that $\lambda_G(x)=\lambda_{HN}(x)\lambda_{G/N}(xN)$. Hence, by the same argument as before, $\alpha(G)=\alpha(HN)\alpha(G/N)$. By induction on $|G|$, we may assume that $G=HN$, i.\,e. $G$ is a $\pi$-nilpotent group. 

Next, we reduce to the case where $H$ is cyclic. Under this assumption, the result holds for arbitrary groups.

\begin{Prop}\label{cyclic}
Let $G$ be a group with a cyclic Hall subgroup $H$. Then \eqref{eq:hall} holds.
\end{Prop}
\begin{proof}
For every generator $h$ of $H$ we have $\lambda_G(h)=1$. Now let $h\in H$ be of order $\frac{n}{e}<n$. Then for every $d\mid e$ there exist exactly $d$ elements $x\in H$ such that $x^d=h$. Hence, the exponent of $\lambda_G(h)$ in \eqref{eq:hall} is \[\sum_{d\,\mid\, e}d\frac{n}{d}\mu(d)=n\sum_{d\,\mid\, e}\mu(d)=0.\qedhere\]
\end{proof}

For the general case, let $\mathcal{Z}$ be the set of cyclic subgroups of $H$. Instead of running over all elements of $H$, we run over $Z\in\mathcal{Z}$ and then over $z\in Z$. In order to track multiplicity, we use the Möbius function $\mu$ of the lattice $\mathcal{Z}$. Since the subgroups of a cyclic group of order $d$ are in bijection to the divisors of $d$, we have $\mu(Z,W)=\mu(|W/Z|)$ whenever $Z\le W$ and $\mu(Z,W)=0$ otherwise. We define $f(Z):=\sum_{W\in\mathcal{Z}}\mu(Z,W)$.
Recall the inversion formula for Euler's totient function:
\begin{equation}\label{eq:euler}
\phi(n)=\sum_{d\,\mid\, n}\frac{n}{d}\mu(d).
\end{equation}

For $W\in\mathcal{Z}$, let $[W]$ be the set of generators of $W$. For any function $\gamma:G\to\ZZ$, we have 
\[\prod_{w\in W}\gamma(w)=\prod_{Z\le W}\prod_{z\in[Z]}\gamma(z).\] 
By Möbius inversion, it follows that $\prod_{w\in[W]}\gamma(w)=\prod_{Z\le W}\prod_{z\in Z}\gamma(z)^{\mu(Z,W)}$ and 
\begin{equation}\label{eq:moebius}
\prod_{x\in H}\gamma(x)=\prod_{W\in\mathcal{Z}}\prod_{w\in[W]}\gamma(w)=\prod_{W\in\mathcal{Z}}\prod_{Z\le W}\prod_{z\in Z}\gamma(z)^{\mu(Z,W)}=\prod_{Z\in\mathcal{Z}}\Bigl(\prod_{z\in Z}\gamma(z)\Bigr)^{f(Z)}.
\end{equation}
Counting the number of factors on both sides also reveals that 
\begin{equation}\label{eq:fZ}
|H|=\sum_{Z\in\mathcal{Z}}|Z|f(Z).
\end{equation}

Assuming $G=NH$, we have
\begin{equation}\label{eq:lam}
\lambda_G(x)=|\C_N(x):\C_N(H)|=|\C_N(x):\C_N(Z)||\C_N(Z):\C_N(H)|=\lambda_{ZN}(x)|\C_N(Z):\C_N(H)|
\end{equation}
for $x\in Z\in\mathcal{Z}$.
Now we can put everything together and apply \autoref{cyclic}:
\begin{align*}
\alpha(G)&\overset{\eqref{eq:moebius}}{=}\prod_{Z\in\mathcal{Z}}\Bigl(\prod_{z\in Z}\prod_{d\,\mid\,n}\lambda_G(z^d)^{\frac{n}{d}\mu(d)}\Bigr)^{f(Z)}
\overset{\eqref{eq:lam}}{=}\prod_{Z\in\mathcal{Z}}\Bigl(\alpha(NZ)\prod_{d\,\mid\,n}|\C_N(Z):\C_N(H)|^{|Z|\frac{n}{d}\mu(d)}\Bigr)^{f(Z)}\\
&\overset{\eqref{eq:euler}}{=}\prod_{Z\in\mathcal{Z}}|\C_N(Z):\C_N(H)|^{\phi(n)|Z|f(Z)}
\overset{\eqref{eq:fZ}}{=}\Bigl(|\C_N(H)|^{-|H|}\prod_{Z\in\mathcal{Z}}|\C_N(Z)|^{|Z|f(Z)}\Bigr)^{\phi(n)}.
\end{align*}
At this point, the claim follows from Wielandt's formula~\cite[Satz~2.3]{WielandtFix}, which we prove for sake of self-containment.

\begin{Thm}[Wielandt]\label{wielandt}
Let $H$ be a group acting coprimely on a group $N$. Then 
\[|\C_N(H)|^{|H|}=\sum_{Z\in\mathcal{Z}}|\C_N(Z)|^{|Z|f(Z)}.\]
\end{Thm}
\begin{proof}
Since Wielandt's paper is hard to follow (even for a German native speaker), we use some modern ingredients. We consider $N$ as a $H$-set via conjugation. By a theorem of Hartley--Turull~\cite[Lemma~2.6.2]{HartleyTurull} (see also \cite[Satz~9.20]{SambalePG}), there exist a direct product $A$ of elementary abelian groups and an isomorphism of $H$-sets $\phi:N\to A$, i.\,e. $\phi(n^h)=\phi(n)^h$ for $n\in N$ and $h\in H$. It follows that $\phi(\C_N(Z))=\C_A(Z)$ for $Z\in\mathcal{Z}$. Thus, we may replace $N$ by $A$. Then $N$ decomposes into its (characteristic) Sylow subgroups $N=N_1\times\ldots\times N_k$. Since $\C_N(Z)=\C_{N_1}(Z)\times\ldots\times\C_{N_k}(Z)$, we may assume further that $N=N_1$ is elementary abelian. Thus, $N$ is an $\FF_pH$-module for some prime $p$ not dividing $|H|$. The corresponding Brauer character $\chi:H\to\CC$ can be regarded as an ordinary character since $|H|$ is coprime to $p$. We further extend $\chi$ to the complex group algebra $\CC H$.
For $S\subseteq H$ let $S^+:=\sum_{s\in S}s\in\CC H$. 
The additive version of \eqref{eq:moebius} reads 
\[H^+=\sum_{Z\in\mathcal{Z}}f(Z)Z^+.\] 
By the first orthogonality relation, $\chi(Z^+)=|Z|[\chi_Z,1_Z]$, where $[\chi_Z,1_Z]$ is the multiplicity of the trivial character $1_Z$ as a constituent of the restriction $\chi_Z$. On the other hand, we have $|\C_N(Z)|=p^{[\chi_Z,1_Z]}$. It follows that
\[|\C_N(H)|^{|H|}=p^{|H|[\chi,1_H]}=p^{\chi(H^+)}=p^{\sum_{Z\in\mathcal{Z}}|Z|[\chi_Z,1_Z]f(Z)}=\prod_{Z\in\mathcal{Z}}|\C_N(Z)|^{|Z|f(Z)}.\qedhere\]
\end{proof}

The proof of \autoref{wielandt} relies on the Feit--Thompson theorem to guarantee that $H$ or $N$ is solvable. 
In comparison, the proof of \autoref{NR} does not require representation theory, but uses the fact that Sylow subgroups are nilpotent.

\section{The proof of \autoref{additive}}

A $\pi$-element $x\in G$ lies in a Hall $\pi$-subgroup $H$ if and only if $x\in\N_G(H)$. Hence, the map $\lambda_G:H\to\ZZ$ is the permutation character of the conjugation action of $H$ on the set $\Hall_\pi(G)$ of all Hall $\pi$-subgroups of $G$ (we do not assume that these subgroups are conjugate in $G$). We use the following recipe to construct a related character.

\begin{Thm}\label{thmsym}
Let $\chi$ be a character of a finite group $H$ and let $\alpha$ be a character of a permutation group $A\le S_n$. For $a\in A$ let $c_i(a)$ be the number of cycles of $a$ of length $i$. Then the map $\chi_\alpha:H\to\CC$ with 
\[\chi_\alpha(h)=\frac{1}{|A|}\sum_{a\in A}\alpha(a)\prod_{i=1}^n\chi(h^i)^{c_i(a)}\]
for $h\in H$ is a character or the zero map.
\end{Thm}
\begin{proof}
See \cite[Theorem~7.7.7 and Eq.~(7.7.9)]{KerberAFGA}.
\end{proof}

We assume the notation of \autoref{additive} and choose a cyclic subgroup $A\le S_n$ generated by a cycle of length $n$. Let $\alpha\in\Irr(A)$ be a faithful character. For $B\le A$ of order $d$, $\sum_{b\in[B]}\alpha(b)$ is the sum of the primitive roots of unity of order $d$. A simple Möbius inversion shows that this sum equals $\mu(d)$. Moreover, every $b\in[B]$ is a product of $n/d$ disjoint cycles of length $d$. Thus, $c_d(b)=n/d$ and $c_i(b)=0$ for $i\ne d$.

We apply \autoref{thmsym} with $\chi=\lambda_G$. For $h\in H$ we compute
\[\chi_\alpha(h)=\frac{1}{n}\sum_{d\,\mid\, n}\sum_{\substack{a\in A\\|\langle a\rangle|=d}}\alpha(a)\lambda_G(h^d)^{\frac{n}{d}}=\frac{1}{n}\sum_{d\,\mid\, n}\mu(d)\lambda_G(h^d)^{\frac{n}{d}}.\]
Taking the scalar product of $\chi_\alpha$ and the trivial character of $H$, shows that
\[\beta_G(H):=\frac{1}{n^2}\sum_{h\in H}\sum_{d\,\mid\, n}\mu(d)\lambda_G(h^d)^{\frac{n}{d}}\]
is a non-negative integer. 
This confirms the first part of \autoref{additive}.

If $H\unlhd G$, then $\lambda_G(h)=1$ for all $h\in H$ and it follows that $\beta_G(H)=0$ unless $H=1$ (where $\beta_G(H)=1$). Now assume that $H$ is not normal in $G$. In particular, $n>1$.
Let $t:=|\Hall_\pi(G)|$. If $t=2$, then $|G:\N_G(H)|=2$ and $\N_G(H)\unlhd G$. But this would imply that $H^g=\pcore_{\pi}(\N_G(H))=H$ for every $g\in G$. Hence, $t\ge 3$. 

Next we investigate the contribution of $d=1$ to $\beta_G(H)$. Note that $\lambda_G(h)^n$ is the number of fixed points of $h$ on $\Hall_\pi(G)^n$, acting diagonally. The number of orbits of $H$ on $\Hall_\pi(G)$ is at most $t^n/n$. 
Using Burnside's lemma and the trivial estimate $\lambda_G(h^d)\le t$ for $d\ne 1$ and $h\in H$, we obtain
\[
n\beta_G(H)\ge \frac{1}{n}\sum_{h\in H}\lambda_G(h)^n-\frac{1}{n}\sum_{h\in H}\sum_{1\,\ne\, d\,\mid\, n}t^{\frac{n}{d}}\ge \frac{t^n}{n}-\sum_{1\,\ne\, d\,\mid\, n}t^{\frac{n}{d}}.
\]
It suffices to show that
\[n\sum_{1\,\ne\, d\,\mid\, n}t^{\frac{n}{d}}<t^n.\]
If $n$ is a prime, this reduces to $nt<t^n$ and we are done as $t\ge 3$. If $n=4$ and $t=3$, the claim can be checked directly. In all other cases, one can verify that $n\le t^{\frac{n}{2}-1}$ and
\[n\sum_{1\,\ne\, d\,\mid\, n}t^{\frac{n}{d}}\le \sum_{k=0}^{n-1}t^k=\frac{t^n-1}{t-1}<t^n.\]
This finishes the proof.

We remark that the degree $\chi_\alpha(1)=\frac{1}{n}\sum_{d\mid n}\mu(d)t^{n/d}$ has several interesting interpretations. For instance, if $t$ is a prime power, then $\chi_\alpha(1)$ is the number of irreducible polynomials of degree $n$ over the finite field $\FF_t$ (see \cite[Corollary~10.2.3]{Roman}). 
In general, $\chi_\alpha(1)$ is the rank of the $n$-th quotient of the lower central series of a free group of rank $t$ (see \cite[Theorem~5.11 and Corollary~5.12]{MKS}). 

We now give an interpretation of $\beta_G(H)$ in a special case. Suppose that $H$ is nilpotent with regular Sylow subgroups (for all primes). Then the $d$-powers form a subgroup $H^d$, and for every $h\in H^d$ there exist exactly $|H:H^d|$ elements $x\in H$ with $x^d=h$ (see \cite[Hauptsatz~III.10.5 and Satz~III.10.6]{Huppert}). Burnside's lemma applied to $H^d$ acting on $\Hall_\pi(G)^{n/d}$ yields
\[\beta_G(H)=\frac{1}{n}\sum_{d\,\mid\, n}\mu(d)f_{n/d}(H^d),\]
where $f_{n/d}(H^d)$ is the number of orbits of $H^d$ on $\Hall_\pi(G)^{n/d}$.

Finally, we comment on a curiosity. In order to produce similar quantities as $\beta_G(H)$, we may replace $\lambda_G$ by any character of any finite group $G$. For instance, let $\tau$ be the conjugation character of $G=A_5$ on $\Syl_3(G)$. We compute
\[\frac{1}{60^2}\sum_{d\,\mid\, 60}\mu(d)\sum_{g\in G}\tau(g^d)^{\frac{n}{d}}=277777777777777777777777777773333333332754803832758090933.\]
Similar curious numbers arise from $G\in\{S_5,\PSL(2,9),\PGL(2,9)\}$. This can be explained by the presents of large powers of $\tau(1)=10$.


\begin{thebibliography}{1}

\bibitem{HartleyTurull}
B. Hartley and A. Turull, \textit{On characters of coprime operator groups and
  the {G}lauberman character correspondence}, J. Reine Angew. Math.
  \textbf{451} (1994), 175--219.

\bibitem{Huppert}
B. Huppert, \textit{Endliche {G}ruppen. {I}}, Grundlehren der Mathematischen
  Wissenschaften, Vol. 134, Springer-Verlag, Berlin, 1967.

\bibitem{KerberAFGA}
A. Kerber, \textit{Applied finite group actions}, Algorithms and Combinatorics,
  Vol. 19, Springer-Verlag, Berlin, 1999.

\bibitem{MKS}
W. Magnus, A. Karrass and D. Solitar, \textit{Combinatorial group theory},
  Dover Publications, Inc., Mineola, NY, 2004.

\bibitem{NR}
G. Navarro and N. Rizo, \textit{A {B}rauer-{W}ielandt formula \textup{(}with an
  application to character tables\textup{)}}, Proc. Amer. Math. Soc.
  \textbf{144} (2016), 4199--4204.

\bibitem{Roman}
S. Roman, \textit{Field theory}, Graduate Texts in Mathematics, Vol. 158,
  Springer, New York, 2006.

\bibitem{SambalePG}
B. Sambale, \textit{Endliche {P}ermutationsgruppen}, Springer Spektrum,
  Wiesbaden, 2017.

\bibitem{WielandtFix}
H. Wielandt, \textit{Beziehungen zwischen den {F}ixpunktzahlen von
  {A}utomorphismengruppen einer eindlichen {G}ruppe}, Math. Z. \textbf{73}
  (1960), 146--158.

\end{thebibliography}
\end{document}